\documentclass[12pt]{article}

\usepackage{amssymb,amsmath,latexsym,amsthm}

\newtheorem{thm}{Theorem}[section]
\newtheorem{lemma}[thm]{Lemma}

\newtheorem{cor}[thm]{Corollary}
\newtheorem{remark}[thm]{Remark}

\theoremstyle{definition}
\newtheorem{defn}[thm]{Definition}

\numberwithin{equation}{section}

\begin{document}

\title{ On the prime divisors of elements of a $D(-1)$ quadruple 
}
\author{Anitha  Srinivasan
\thanks{Department of Mathematics, Saint Louis University-Madrid campus, 
Avenida del Valle 34, 28003 Madrid, Spain}
}
\date{}
\maketitle

\begin{abstract}
 In \cite{FFM} it was shown that  
if $\{1, b, c, d\}$ is a $D(-1)$ quadruple with 
 $b<c<d$ and $b=1+r^2$, then 
$r$ and $b$ are not of the form $r=p^k, \hskip1mm 
 r=2p^k,\hskip1mm b=p$ or $ b=2p^k$, where
 $p$ is an odd prime and $k$ is a positive integer.
We show that an identical result holds for $c=1+s^2$,
 that is,  
the cases $s=p^k,\hskip2mm s=2p^k,\hskip2mm c=p$ and $ c=2p^k$
do not occur for the $D(-1)$ 
quadruple given above. For the integer $d=1+x^2$, we show 
that $d$ is not prime and that $x$ 
is divisible by at least two distinct odd primes.
Furthermore, we present several infinite families of integers 
$b$ such that the 
$D(-1)$ pair $\{1, b\}$ cannot
be extended to a $D(-1)$ quadruple.
  For instance, we show that 
if $r=5p$ where $p$ is an  odd prime, then the 
 $D(-1)$ pair $\{1, r^2+1\}$ cannot
be extended to a $D(-1)$ quadruple. 
\end{abstract}

 \noindent
 {\footnotesize{\textbf{ AMS Subject
Classification}}: 11D09, 11R29, 11E16.
  }
\bigskip
\section{Introduction}
Let $n$ be a non zero integer. 
 A $D(n)$ tuple is a set of positive integers such that 
if $a, b$ are any two elements from this set, then 
$ab+n=k^2$ for some integer $k$.  
We will look at the case $n=-1$. The cases $n=1$ and $n=4$ 
have been studied in great detail and still continue to be 
areas of active research. For more details on these cases the
reader may consult the references given in \cite{FFM}.

In the case of $n=-1$, it has been conjectured that there is no 
$D(-1)$ quadruple. 
 The first significant progress 
was made by Dujella and Fuchs \cite{DF},  
who showed that if $\{a, b, c, d\}$ is a 
$D(-1)$ quadruple with $a<b<c<d$, then $a=1$. 
Subsequently, Dujella {\it et. al.} \cite{DFF} proved that there
 are only a finite number of such quadruples. 
Filipin and Fujita (\cite{FF}) showed that 
 if $\{1, b, c\}$ is $D(-1)$ triple
with $b<c$, 
then there exist at most two $d$'s such that 
$\{1, b, c, d\}$ is a $D(-1)$ quadruple. 

Recently, Filipin {\it et al.} \cite{FFM} 
 showed that if $b=r^2+1$, then in each of the cases $r=p^k,
\hskip 2mm r=2p^k, \hskip2mm b=p$ and  $ b=2p^k$, where
 $p$ is an odd prime and  
$k$ is a positive integer,  
the $D(-1)$ pair $\{1, b\}$ cannot be extended to a
$D(-1)$ quadruple $\{1, b, c, d\}$ with $b<c<d$. 
The existence of a $D(-1)$ quadruple is closely 
related to the existence of 
solutions of quadratic diophantine equations 
of the type 
$X^2-(1+Z^2)Y^2=Z^2$.
The above result of \cite{FFM} 
is a corollary  of an extremely useful result
proved therein (\cite[Theorem 1.1]{FFM}) or Lemma 4.1 for a partial result) on the 
equivalence of certain solutions of the diophantine 
equation 
$X^2-(1+r^2)Y^2=r^2$.
We use this result in conjunction with our
methods from class groups to prove our theorems.
 Our first theorem 
shows that the result in \cite{FFM} mentioned above also holds for  
 $c$ and $d$. (Note that $d$ is known to be odd and $b, c $
 and $d$  cannot be
of the form $p^k$ with $k>1$ and $p$ prime.) 
While our proof of Theorem 1.1 below for $c$ and $s$ 
serves also to prove the identical result for $b$ and $r$ 
given in \cite{FFM},  
the proof in \cite{FFM} for this case 
 does not work for $c$ and $s$ as it is assumed therein 
that $b<c$.  
\begin{thm}  
Let $\{1, b, c, d\}$ with 
$1<b<c<d$ be a $D(-1)$ quadruple where 
  $c=1+s^2 $. Let  
$p$  be  an odd prime and $k$ a positive integer. Then
the cases $c=p,\hskip2mm d=p,\hskip2mm c=2p^k,\hskip2mm 
s=p^k $ and $s=2p^k$ do not occur. 
Moreover, if $d=1+x^2$, then $x$ is divisible by at least two 
distinct odd primes.
\end{thm}
In the case of a product of two odd primes, we have the 
following result.
\begin{thm} Let $\{1, b, c, d\}$ be a $D(-1)$ quadruple 
with $b<c<d$. If $b=1+r^2$ and $r=pq$,  
 where $p$ and $q$ are distinct odd primes, then 
$p, q>r^{\frac{1}{4}}$.
\end{thm}
\begin{cor} Suppose that $\alpha$ is a positive integer such that 
for each $r\le\alpha$ 
 the $D(-1)$ pair 
$\{1, r^2+1\}$ cannot be extended to a $D(-1)$ quadruple. 
 Then  for each odd prime  
$p\le \alpha^{\frac{1}{4}}$ and any odd prime $q\ne p$ the 
$D(-1)$ pair 
$\{1, (pq)^2+1\}$ cannot be extended to a $D(-1)$ quadruple. 
 
\end{cor} 
\begin{remark} To illustrate a concrete case of the above corollary, note 
that one may verify that if $r\le\alpha=5^4$, then the $D(-1)$ pair
$\{1, r^2+1\}$ cannot be extended to a $D(-1)$ quadruple.  
 Hence if $p=5$, then for  
$r=pq>5^4$ we have  
 $p=5= \alpha^{\frac{1}{4}}$ and therefore 
by Corollary 1.3 the $D(-1)$ pair $\{1, (5q)^2+1\}$ cannot be extended 
to a $D(-1)$ quadruple for any odd prime $q$. 
\end{remark}
\begin{thm} Let $r=P\phi$, where $P$ is prime and $\phi<r^{\frac{1}{4}}$. Then 
there is no 
$D(-1)$ triple $\{1, r^2+1, s^2+1\}$ with $\gcd(r, s)=1$.
\end{thm}
  We provide an entirely new approach  
via  the theory of 
binary quadratic forms and the class group to study this problem. 
This is possible as the existence of a $D(-1)$ triple is intimately
connected to the representations of integers by certain
binary quadratic forms 
and hence to the class group.
\section{Binary quadratic forms and the class group}
 In this section we present the basic theory of
binary quadratic forms. An excellent and delightful reference
for this topic is \cite{Ri}, where in particular, the reader may 
consult Sections 4 to 7 and 
Section 11 for the material presented here.

A {\sl primitive binary quadratic form} $f=(a, b, c)$ of discriminant $d$ is
a function $f(x, y)=ax^2+bxy+cy^2$, where $a, b,c $ are integers
with $b^2-4 a c=d$ and $\gcd(a, b, c)=1$. Note that the integers 
$b$ and $d$ have the same parity. All forms considered here
are primitive binary quadratic forms and henceforth we shall
refer to them simply as forms. 

Two forms $f$ and $f'$ are said to be {\it equivalent}, written as
$f\sim f'$,  if for some
$A=\begin{pmatrix} \alpha &\beta \\ \gamma & \delta \end{pmatrix}
\in SL_2(\mathbb Z)$ (called a transformation matrix),  we have
$f'(x,y)=f(\alpha x+\beta y, \gamma x+\delta y)
=(a',b',c')$, where 
the coefficients $a', b', c'$ are given by
\begin{equation}
a'=f(\alpha, \gamma),\hskip2mm
b'=2(a\alpha\beta+c\gamma\delta)+b(\alpha\delta+\beta\gamma),\hskip2mm
c'=f(\beta, \delta).
\end{equation}
 It is easy to see
that $\sim$ is an equivalence relation on the set of forms of
discriminant $d$. The equivalence classes form an abelian group
called the  {\it class group} with group law given by composition of
forms (see Definition 2.2).

The {\it identity form} is defined as the form $(1,0,\frac{-d}{4})$
or $(1, 1, \frac{1-d}{4})$ depending on whether $d$ is even or odd,
respectively. 
 The {\it inverse} of
$f=(a, b, c)$, denoted by $f^{-1}$, is given by $(a,-b,c).$

A form $f$ is said to represent an integer $m$ if there exist 
integers $x$ and $y$ such that $f(x,y)=m$. If $\gcd(x, y)=1$, we 
call the represention a primitive one.
Observe that equivalent forms primitively represent the same set
of integers.

We put together some basic facts about forms of 
discriminant $d$ in the following lemma.
\begin{lemma} The following hold for forms of discriminant $d$.  
\begin{enumerate}
\item
An integer $n$ is primitively represented by a form $f$ 
 if and only if $f\sim (n, b, c)$ for some integers $b, c$. 
\item If $f=(n, b, c)$ and $f'=(n, b', c')$ are two forms such that
$b\equiv b'\mod 2n$, then $f\sim f'$. 
\item Let $n$ with $\gcd(n, 2d)=1$ be an integer primitively represented 
by some form and let $w(n)$ be the number of 
distinct primes dividing $n$. Then there are 
$2^{w(n)-1}$ forms $(n,  b, c)$ where $1\le b\le 2n$.  
\end{enumerate}
\end{lemma}


In the following definition we present the formula for the
 composition of forms that gives the group multiplication for the class group.

 Let $f_1=(a_1, b_1, c_1) \text{ and }
f_2=(a_2, b_2, c_2)$ be two binary quadratic forms
of discriminant $d$.
\begin{defn} 
Let $g=gcd(a_1, a_2, (b_1+b_2)/2)$ and let $v_1, v_2, w $ be integers
such that $$v_1a_1+v_2a_2+w(b_1+b_2)/2=g.$$ If we define $a_3$ and
$b_3$ as
\[
\begin{split} a_3&=\frac{a_1a_2}{g^2},\\
b_3&= b_2+2\,\frac{a_2}{g}\,\left(\frac{b_1-b_2}{2}\,\,v_2-c_2w\right)
\mod\hskip2mm 2a_3 ,
\end{split}
\]
then the composition of the forms $(a_1, b_1, c_1)$ and
$(a_2, b_2, c_2)$ is the form $(a_3, b_3, c_3)$, where $c_3$ is
computed using the discriminant equation 
$b_3^2-4a_3 c_3=d$.
Note that $b_3$ is taken modulo $2a_3$ because of Lemma 2.1, part 2.
\end{defn}

\section{The diophantine equation $x^2-dy^2=n$}

The study of $D(-1)$ quadruples leads to the 
study of forms
$(1, 0, -d)=x^2-dy^2$ of discriminant $4d$. 
 If $(x, y)$ is a primitive 
representation of an integer $n$ by this form (i.e. $x^2-dy^2=n$),
 then there exist integers $\alpha$ and $\beta$ such that 
the matrix 
$
A=\begin{pmatrix}
x & \alpha \\
y & \beta
\end{pmatrix}
$
has determinant $1$.  By 
(2.1) the matrix $A$ transforms the form $(1, 0, -d)$ to a form
$(n, 2b, c)$. 
 Observe that the 
choice of $\alpha$ and $\beta$ is not unique. The following facts 
are easy to verify 
(see for example \cite[Solution of problem 3, Section 7]{Ri}). 
 Any choice of
integers $u, v$ such that $xv-yu=1$ yields a transformation matrix
that takes $(1, 0, -d)$ to a form $(n, 2b', c')$ where $b'\equiv b \mod n$.
Moreover, it may also be verified that if $(n, 2b', c')\sim (1, 0, -d)$
 is a form such that 
$b'\equiv b \mod n$, then there exist integers $u, v$ such that $xv-yu=1$.
Therefore for each primitive representation $(x, y)$ of $n$ by the 
form $(1, 0, -d)$, there corresponds a unique 
integer $b \mod n$. We say 
in this case that the representation $(x, y)$ belongs to $b$.

If two primitive representations $(x, y)$ and $(x', y')$ (of $n$ by 
$(1, 0, -d)$) belong to the
same integer $b$, then it may be verified that 
\begin{equation}
xx'\equiv d y y'\mod n, \hskip 2mm 
xy'\equiv yx'\mod n.
\end{equation}
We call such representations as equivalent. The congruences in 
(3.1) may be used to  
define 
equivalence of  general solutions (that are not necessarily primitive)
as follows.
\begin{defn}
Two solutions $(x, y)$ and $(x', y')$ of 
$X^2-dY^2=n$ are said to be equivalent, written as 
$(x, y)\sim (x', y')$ if the congruences (3.1) are satisfied.  
\end{defn}
The following lemma
 is easy to verify using the theory of class
groups (not presented here).
 It is used by several authors in the study of the current problem, such as
\cite[Lemma 6.2]{FFM}.
\begin{lemma} If  
 $|n|<k$ then there are no primitive solutions 
$(x, y)$ such that $x^2-(k^2+1)y^2=n$.
\end{lemma}
The following result is a useful consequence of the above lemma that we 
use to prove our theorems.
\begin{lemma} Let $k=ff'$ be an odd positive integer 
 such that $1<f<k$.
If $x^2-(k^2+1)y^2=f'^2$ for some coprime 
integers $x$ and $y$,  
then $f'$ is not a prime power. 
\end{lemma}
\begin{proof}
As $x^2-(k^2+1)y^2=f'^2$, 
it follows from Lemma 3.2 that $f<f'$, that is $f'>\sqrt{k}$. 
Moreover, as $(1, 0, -(k^2+1))$ primitively
represents $f'^2$, by Lemma 2.1, part 1, there is a form 
$(f'^2, 2b, c)$ for some integers $b, c$ such that 
$(1, 0, -d)\sim (f'^2, 2b, c)$.

Observe that $F=(f'^2, 2, -f^2)$ is a form of discriminant $4(k^2+1)$.  
If $f'$ is a prime power, then by Lemma 2.1 we have 
 $(f'^2, 2b, c)\sim(f'^2,  2, -f^2)$ 
or
 $(f'^2, 2b, c)\sim(f'^2,  -2, -f^2)$ 
and hence 
as equivalent forms
primitively represent the same integers, the form  
$(1, 0, -d)$ primitively 
represents $-f^2$ (as the forms $(f'^2, \pm 2, -f^2)$ represent $-f^2$ 
via the representation $(0, 1)$),
 which is not possible by Lemma 3.2, 
as $f^2<k$. 
\end{proof}
\section{Proofs of theorems}

The following terminology will hold throughout this section.

Let $\{1, b, c, d\}$ be a $D(-1)$ quadruple with 
$1<b<c<d$.
Let $$b=1+r^2,\hskip2mm c=1+s^2,\hskip2mm d=1+x^2$$ and  
$$bd=1+y^2,\hskip2mm cd=1+z^2,\hskip2mm bc=1+t^2.$$
Then 
\begin{equation}
t^2-(1+r^2)s^2=r^2
\end{equation}
and 
\begin{equation}
t^2-(1+s^2)r^2=s^2.
\end{equation}
Observe that for any positive integer $k$, the 
equation $X^2-(k^2+1)Y^2=k^2$ has the 
inequivalent solutions $(\pm k, 0)$ and $ (k^2+1-k, \pm (k-1))$.
\begin{lemma}{\rm (\cite[Theorem 1.1]{FFM})} The solution $(t, s)$ given in (4.1) 
of $X^2-bY^2=r^2$ is not equivalent to any of the solutions 
$(b-r, \pm (r-1))$ and $(\pm r, 0)$.
\end{lemma}
\begin{lemma}{\rm \cite[Proof of Theorem 1, p. 389]{FF}} If $M={\rm lcm}(r, s)$,  then 
$x\equiv 0\mod M^2$.
\end{lemma}
For the following lemma note that if $(x, y)$ is a primitive 
solution of $X^2-(k^2+1) Y^2=k^2$, then 
$(x, y)\not\sim (x, -y)$. Note also that if the representation
$(x, y)$ belongs to the integer $b$, then $(x, -y)$ belongs to the 
integer $-b$.(See beginning of Section 3.) 
\begin{lemma}
Let $k$ be an odd  positive integer such that there are two 
primitive solutions $(a, b)$ and $(a',b')$ 
to $X^2-(k^2+1)Y^2=k^2$ such that 
$(a, b)\not\sim(a',\pm b')$. Then there exist 
coprime integers $p,q$ both greater than $1$, with $k=pq$ such that 
$p^4$ and $q^4$ are primitively represented by the 
form $X^2-(k^2+1)Y^2$.
\end{lemma}
\begin{proof}
 
The given two primitive solutions to $X^2-(k^2+1)Y^2=k^2$, 
by Lemma 2.1 part 1, give rise to 
 two forms 
$P_1$ and $P_2$ equivalent to $(1, 0, -(k^2+1))$ such that 
 $P_1=(k^2, 2b_1, c_1)$ and $P_2=(k^2, 2b_2, c_2)$, where 
$b_1, b_2$ satisfy 
$b_1\not\equiv \pm b_2 \mod k^2$ (see remark preceding this lemma). 
Observe that  from the discriminant equation we have
$b_1^2\equiv b_2^2 \mod k^2$. 
As 
  $\gcd(k, b_i)=1$, 
there exist
coprime integers
$p ,q$ greater than $1$ with $k=pq$ such that
\begin{equation}
b_1+b_2\equiv 0\mod p^2,\\\\ b_1-b_2\equiv 0\mod q^2.
\end{equation}
Let $I=(p^2, 2b_1, c_1q^2)$ and 
$J=(q^2, 2b_1, c_1p^2)$.
By Definition 2.2 (composition of forms), as $p, q$ are coprime, we have  
$IJ\sim P_1$ and $I^{-1}{J}\sim P_2$ and hence
$$(1, 0, -(k^2+1))\sim P_1\sim IJ\sim P_2\sim I^{-1}J.$$
It follows that 
$$I\sim I^{-1}\sim J,$$
and therefore $I^2\sim J^2\sim (1, 0, -(k^2+1))$. Using Definition 2.2 again,
it is easy to verify that 
$I^2=(p^4, 2\phi, \psi)$ for some integers
$\phi$ and $\psi$ and hence, as $I^2\sim (1, 0, -(k^2+1))$, the 
form $(1, 0, -(k^2+1))$ primitively represents $p^4.$ Similarly 
 $(1, 0, -(k^2+1))$  also primitively represents $q^4$.
\end{proof}

{\bf Proof of Theorem 1.1 }
Suppose that 
$c= p$ for some odd prime $p$.  
From (4.1) we have 
\begin{equation}
t^2-s^2=r^2(1+s^2)
\end{equation}
so that for some decomposition 
$r=r_1 r_2$ we have 
$$\frac{t-s}{r_1^2}\hskip2mm\frac{t+s}{r_2^2}=1+s^2=p.$$
 Therefore 
either 
\begin{equation}
t-s=pr_1^2, \hskip2mm t+s=r_2^2
\label{case1}
\end{equation}
or
\begin{equation}
t-s=r_1^2, \hskip2mm t+s=pr_2^2.
\label{case2}
\end{equation}
In the case of (\ref{case1}) we have 
$$2s+cr_1^2=r_2^2$$
which is not possible 
as $b<r^2<2s+cr_1^2=r_2^2\le r^2.$

Assume now that (\ref{case2}) holds. Then 
\begin{equation}
2s+r_1^2=cr_2^2
\label{rs}
\end{equation}
which is possible only when $r_2=1$,  as 
$2s^2\ge 2s+r^2\ge 2s+r_1^2=cr_2^2>s^2r_2^2.$
If $r_2=1$ then $r_1=r$ 
and from (\ref{case2}) we have 
$t-s=r^2$. 
However this is not possible as $t\equiv s\mod r^2$ implies
by Definition 3.1 (equivalence of solutions) that 
$(t, s)\sim (b-r, 1-r)$, which by Lemma 4.1 is not true. 

The proofs in the cases when $c=2 p^k$ and $d=p$  are similar, 
where in the latter case we work as above with the equation
$y^2-x^2=r^2d$.

Assume now that $s=p^k$. 
From (4.2) we have $t^2-r^2=bs^2$ so that 
  if $\gcd(t+r, t-r)=1$ then 
for some factorization $b=b_1 b_2$ we have either 
\begin{equation}
t-r=b_1 s^2, \hskip 1mm t+r=b_2
\end{equation} or 
\begin{equation}
t-r=b_1, \hskip 1mm t+r=b_2 s^2.
\end{equation}
If (4.8) holds, then 
we have 
$b_1s^2+2r=b_2\le b=1+r^2$ 
which is not possible as $r< s$.
If (4.9) holds, then 
$2r+b_1=b_2 s^2\le 2r+1 +r^2$, which gives  
$b_2=1$ in which case from (4.9) we have  
$s=r+1$ and  $t=r^2+s$. The latter is not possible as seen above in the proof 
of the case when $c=p$.

 We assume now that $\gcd(t+r, t-r)>1$.  
It follows that $t=p^m t_1$ and $r=p^m r_1$ where 
$\gcd(t_1, r_1)=1$ and $1\le m\le k$. Then (4.2) gives 
\begin{equation}
t_1^2-c r_1^2=p^{2k-2m},
\end{equation}
which by Lemma 3.3 is not possible if $m<k$. Hence 
$m=k$, that is $s|r$, which is not true as $r<s$. 
The case when $s=2p^k$ is similarly dealt with.

In the case of the integer $x$, note by Lemma 4.2 that $x$ is 
divisible by the least common multiple of $r$ and $s$. Hence, if
$x$ is not divisible by two distinct odd primes, then we must 
have 
$$r=2^{\alpha} p^m, \hskip3mm s=2^{\beta} p^n, $$ where 
$m, n, \alpha, \beta$ are non negative integers and $p$ is an 
odd prime.
Observe from (4.1) and (4.2) that 
$$\gcd(t, s)=\gcd(t, r)=\gcd(r, s).$$
Assume that $\alpha=\beta=0$ in which case $m<n$, so that 
$r|s$, which is not possible as this implies by Definition 3.1
that the solutions $(r, 0)$ and $(t, s)$ are equivalent, contradicting
Lemma 4.1. 

Assume now that $\alpha>0$ and $\beta=0$. As $r<s$ it follows that 
$m<n$. Hence $\gcd(r, s)=p^m=\gcd(t, s)$. Therefore 
$t=p^m t_1$ and from (4.2) we have 
$$t_1^2-c2^{2\alpha}=p^{2n-2m}$$
which is not possible by Lemma 3.3.

The other cases follow similarly and thus 
we have shown that $x$ is divisible by at least two distinct odd primes. 

{\bf Proof of Theorem 1.2}

We have $n=\gcd(t, s) |r$. Note that $n\ne r$ as if $r|s$, then as 
in the proof above, by Definition 3.1 
the solutions $(r, 0)$ and $(t, s)$ are equivalent, which is not 
the case by Lemma 4.1. 
Therefore, if  $n>1$, then 
$n=p$ or $n=q$. 
 If 
$t=nt_1$ and $s=ns_1$ with $\gcd(t_1, s_1)=1$, 
then from (4.1) we have 
\begin{equation}
t_1^2-bs_1^2=\left(\frac{r}{n}\right)^2 
\end{equation}
which is not possible 
by  Lemma 3.3. 
Thus $n=1$ and  
 we have two primitive solutions 
of $X^2-bY^2=r^2$, namely
$(b-r,  (r-1))$ and $(t, s)$. 
By Lemmas 4.1 and 4.3 it follows that 
$p^4$ and $q^4$ are represented by the 
form $(1, 0, -b)$.  Finally, Lemma 3.2 gives 
$p^4>r$ and $q^4>r$, which yields the desired result.

\hfill $\qed$.

{\bf Proof of Corollary 1.3}

If $b=(pq)^2+1$ and $pq\le\alpha$, then by assumption the
$D(-1)$ pair $\{1, b\}$ cannot be extended to a 
$D(-1)$ quadruple.
 Hence we assume that 
$pq> \alpha$. If 
$p\le \alpha^{\frac{1}{4}}\le (pq)^{\frac{1}{4}}$, then 
by Theorem 1.2 it follows that the 
$D(-1)$ pair $\{1, b\}$ cannot be extended to a 
$D(-1)$ quadruple.

\hfill $\qed$

{\bf Proof of Theorem 1.5}

If $\gcd(r, s)=1$, then $\gcd(t, s)=1$ and  
it follows from (4.1) and Lemma 4.1 that there are two  
primitive solutions of 
$X^2-bY^2=r^2$, 
namely $(b-r,  (r-1))$ and 
$(t,  s)$ that satisfy 
$(b-r, r-1)\not \sim (t, \pm s)$.
 Therefore by Lemma 4.3, 
there exists a factorization 
$r=pq$, where $p$ and $q$ are coprime and both 
greater than $1$, such that 
$p^4$ and $q^4$ are primitively represented by the form 
$(1, 0, -b)$. However as $r=P\phi$, at least one of 
$p$ or $q$ say $p$, divides $\phi<r^{\frac{1}{4}}$ 
and thus $p^4<r$ which is not possible by Lemma 3.2.

\hfill $\qed$

\end{document}